\documentclass{article}
\usepackage{amsfonts}
\usepackage{latexsym}
\usepackage{mathrsfs}
\usepackage{amssymb}
\usepackage{amsmath}
\usepackage{amsthm}
\usepackage{indentfirst}
\usepackage{color}
\usepackage{float}
\usepackage{graphicx}

\allowdisplaybreaks
\newtheorem{theorem}{\color{black}\indent Theorem}[section]
\newtheorem{lemma}{\color{black}\indent Lemma}[section]
\newtheorem{proposition}{\color{black}\indent Proposition}[section]
\newtheorem{definition}{\color{black}\indent Definition}[section]
\newtheorem{remark}{\color{black}\indent Remark}[section]

\begin{document}
\title{\LARGE\bf Threshold results for the existence of global and blow-up solutions to Kirchhoff equations with arbitrary initial energy}
\author{Yuzhu Han$^\dag$ \qquad Qingwei Li}
 \date{}
 \maketitle

 \footnotetext{\hspace{-1.9mm}$^\dag$Corresponding author.\\
  Email addresses: yzhan@jlu.edu.cn(Y. Han).

\thanks{
$^*$The project is supported by NSFC (11401252),
by Science and Technology Development Project of Jilin Province (20150201058NY,20160520103JH) and
by the project of The Education Department of Jilin Province (2015-463).}}

\begin{center}
{\it\small School of Mathematics, Jilin University,
 Changchun 130012, P.R. China}
\end{center}

\date{}
\maketitle

{\bf Abstract}\ In this paper we will apply the modified potential well method and variational method to the study of the long time behaviors of solutions to a class of parabolic equation of Kirchhoff type. Global existence and blow up in finite time of solutions will be obtained for arbitrary initial energy.
To be a little more precise, we will give a threshold result for the solutions to exist globally
or to blow up in finite time when the initial energy is subcritical and critical, respectively.
The decay rate of the $L^2(\Omega)$ norm is also obtained for global solutions in these cases. Moreover,
some sufficient conditions for the existence of global and blow-up solutions are also derived when the initial energy is supercritical.



\section{Introduction}
\setcounter{equation}{0}

In the past decades, more and more attention has been devoted to the study of Kirchhoff type problems for their contributions to the modeling
of many physical and biological phenomena. These problems are closely related to the following hyperbolic equation
\begin{equation}\label{1.2}
\rho\frac{\partial^2u}{\partial t^2}-(\frac{P_0}{h}+\frac{E}{2L}\int_0^L|\frac{\partial u}{\partial x}|^2\mathrm{d}x)\frac{\partial^2u}{\partial x^2}=0,
\end{equation}
which was first presented by Kirchhoff \cite{kirchhoff} in 1883 to describe the transversal oscillations of a stretched string,
where the subsequent change in string length caused by oscillations was taken into account.
The parameters in \eqref{1.2} have the following physical interpretations:
\begin{eqnarray*}
&&L: \ \mbox{the length of the string}; \ \ \ \ \ \ \ \ \ \ \ \ \ \ \ \ \ \ \ \ \ \ \ \ \ \ \ \ \ \ \ \ \ \ \ \ \ \ \ \ \ \ \ \ \ \ \ \ \ \ \ \ \ \\
&&h: \ \mbox{the area of cross-section};\\
&&\rho: \ \mbox{the mass density};\\
&&P_0: \ \mbox{the initial tension};\\
&&E: \ \mbox{the Young modulus of the material}.
\end{eqnarray*}
It was mainly after the work of Lions \cite{lions1977}, where a method of functional analysis was proposed to deal with these kind of problems,
that the existence, uniqueness and regularities of solutions to Kirchhoff type equations were well studied by various authors.
Interested reader may refer to, for example, \cite{Shibata1994,Spagnolo1992,Nishihara1999} and the references therein for such results.

The following Kirchhoff type equation is an extension of the classical D'Alembert wave equation for free vibrations of elastic strings
(see \cite{Ghisi2008})
\begin{equation}\label{1.3}
\varepsilon u^{\varepsilon}_{tt}+u^{\varepsilon}_t-M(\int_{\Omega}|\nabla u^{\varepsilon}|^2\mathrm{d}x)\Delta u^{\varepsilon}=f(x,t).
\end{equation}
Formally, taking $\varepsilon=0$, \eqref{1.3} becomes a parabolic equation of Kirchhoff type
\begin{equation}\label{1.4}
u_t-M(\int_{\Omega}|\nabla u|^2\mathrm{d}x)\Delta u=f(x,t).
\end{equation}
Problem \eqref{1.4} can be used to describe the motion of a nonstationary
fluid or gas in a nonhomogeneous and anisotropic medium, and the nonlocal term $M$ appearing in \eqref{1.4} can describe a possible change in the global state of the fluid or gas caused by its motion in the considered medium.
The questions of existence, uniqueness and asymptotic behavior of solutions to \eqref{1.4} have been obtained by Chipot et.al \cite{Chipot2003}.
Nonlocal effects also find their application in biological systems.
When the diffusion coefficient $M$ in \eqref{1.4} depends on the integral of $u$ on the entire domain, i.e. on $\int_{\Omega}u(x,t)\mathrm{d}x$,
\eqref{1.4} can be used to describe the growth and movement of a particular species (for instance of bacteria), where
$u$ could describe the density of a population subject to spreading.

In this article, we are concerned with the following initial boundary value
problem for a class of Kirchhoff type parabolic equation with a nonlinear term
\begin{equation}\label{1.1}
\begin{cases}
u_t-M(\int_{\Omega}|\nabla u|^2\mathrm{d}x)\Delta u=|u|^{q-1}u, &(x,t)\in \Omega\times(0,T),\\
u=0, &(x,t)\in \partial \Omega\times(0,T),\\
u(x,0)=u_0(x), & x\in\Omega.
\end{cases}
\end{equation}
Here the diffusion coefficient $M(s)=a+bs$ with the parameters $a, b$ being positive so that $M$ is chosen in accordance with its original meaning,
$\Omega\subset\mathbb{R}^n(n\geq1)$ is a bounded domain with smooth boundary $\partial\Omega$, $3<q\leq2^*-1$, where $2^*$ is the Sobolev conjugate
of $2$, i.e. $2^*=+\infty$ for $n=1,2$ and $2^*=\dfrac{2n}{n-2}$ for $n\geq 3$. Moreover, $u_0\in H_0^1(\Omega)$.

By introducing a family of potential wells, we will show the invariance of some sets and give a threshold result for the
solutions to exist globally or to blow up in finite time when the initial energy is subcritical or critical.
The decay rate of the $L^2(\Omega)$ norm of solutions are obtained for these cases.
Moreover, by using variational methods, we also give some sufficient conditions for the existence of global and blow-up solutions
for supercritical initial energy.

It was D. H. Sattinger \cite{Sattinger1968} who first proposed the potential well method in 1968 when dealing with a class of nonlinear
hyperbolic initial boundary value problem
\begin{eqnarray}\label{1.5}
\begin{cases}
u_{tt}-\nabla^2u+f(x, u)=0, & (x, t)\in \ \Omega\times(0,T),\\
u=0, & (x, t)\in \ \partial \Omega\times(0,T),\\
u(x,0)=U(x), \ u_t(x,0)=V(x) & \ \ \ \ \ x \in \ \Omega.
\end{cases}
\end{eqnarray}
Instead of a dynamical system, it utilizes a functional  $J(u)$ in an appropriate Sobolev space.
Suppose that $J$ has a local minimum at $u=U(x)$.
A potential well is a region near the locally minimal potential energy. Solutions starting
inside the well are global in time, and the energy is nonincreasing in time.
Solutions starting outside the well and at an unstable point blow up in finite time.
Since then many authors \cite{Ikehata1996,Lions1969,YCLiu2003,YCLiu2006,Payne1975,Tsutsumi2003} studied the global existence and nonexistence of solutions of initial  boundary value problem for various nonlinear evolution equations by using potential well method, a typical one of which is Payne and Sattinger's work \cite{Payne1975}. Later, Liu and his cooperators\,\cite{YCLiu2003,YCLiu2006} generalized and improved Payne and Sattinger's results
by introducing a family of potential wells which include the known potential well as a special case. By using the improved method they not only gave a threshold result of global existence and nonexistence of solutions, but also obtained the vacuum isolating of solutions. Furthermore, they
proved the global existence of solutions with critical initial conditions.

However, to the best of our knowledge,
there have been few works concerned with global existence, blow-up and extinction for the nonlinear parabolic equations with the nonlocal
term $-(a+b\|\nabla u(x,t)\|_2^2)\triangle u$.
A difficulty arising from Problem \eqref{1.1} is the nonlinearity of the nonlocal term,
since one usually can not deduce from $u_n\rightharpoonup u$ in $H_0^1(\Omega)$ the convergence
$\|\nabla u_n\|_2\rightarrow\|\nabla u\|_2$.
Inspired by some ideas from \cite{YCLiu2006,Qu2016,RZXu2009,RZXu2013}, we combine the modified potential well method with the classical Galerkin's method and energy estimates to prove the existence of global weak solutions. Here some tricks arising from $S_+$ operator will be of great help in proving $\|\nabla u_n\|_2\rightarrow\|\nabla u\|_2$.
In addition, by applying the concavity arguments introduced by Levine \cite{Levine1973} together with the properties of potential wells, we obtain the result of blow-up in finite time of solutions for subcritical and critical initial energy.
Moreover, we also give some sufficient conditions for the existence of the global and blow-up solutions with supercritical initial energy,
and show that there exists $u_0$ such that the initial energy $J(u_0)$ is arbitrarily large,
while the corresponding solution $u(x,t)$ of Problem \eqref{1.1} with $u_0$ as initial datum blows up in finite time..

The rest of this paper is organized as follows. In Section 2,
we present some notations, definitions, functionals and sets as well as some lemmas concerning their basic properties.
Sections 3 and 4 will be devoted to the cases $J(u_0)<d$ and $J(u_0)=d$, respectively.
In Section 5, we give some sufficient conditions for the existence of global and blow-up solutions of \eqref{1.1}
when $J(u_0))>d$. Here $J(u)$ is the Lyapunov functional corresponding to \eqref{1.1} that will be introduced in Section 2.

\par
\section{Preliminaries}
\setcounter{equation}{0}

Throughout this paper, we denote by $\|\cdot\|_{2}$ the $L^2(\Omega)$ norm and $(\cdot, \cdot)$ the inner product in $L^2$.
We will equip $H_0^1(\Omega)$ with the norm $\|u\|_{H_0^1(\Omega)}=\|\nabla u\|_2$, which is equivalent to the standard one
due to Poincar\'{e}'s inequality.
In order to state our main results precisely, we first introduce some notations and definitions of some functionals and sets,
and then investigate their basic properties. For $u\in H_0^1(\Omega)$, set
\begin{eqnarray*}
&&J(u)=\frac{a}{2}\|\nabla u\|_2^2+\frac{b}{4}\|\nabla u\|_2^4-\frac{1}{q+1}\|u\|^{q+1}_{q+1},\\
&& \ I(u)=a\|\nabla u\|_2^2+b\|\nabla u\|_2^4-\|u\|^{q+1}_{q+1},
\end{eqnarray*}
and the Nehari manifold
\begin{eqnarray*}
&&\mathcal{N}=\{u\in H_0^1(\Omega)| \ I(u)=0, \ \|\nabla u\|_2\neq0\}. \ \ \ \ \ \ \
\end{eqnarray*}
The potential well and its corresponding set are defined respectively by
\begin{eqnarray*}
&&W=\{u\in H_0^1(\Omega)| \ I(u)>0, \ J(u)<d\}\cup\{0\},\\
&&V=\{u\in H_0^1(\Omega)| \ I(u)<0, \ J(u)<d\},
\end{eqnarray*}
where
\begin{eqnarray*}
&&d=\inf_{0\neq u\in H_0^1(\Omega)}\sup_{\lambda\geq0}J(\lambda u)=\inf_{u\in \mathcal{N}}J(u)
\end{eqnarray*}
is the depth of the potential well $W$.

\begin{lemma}\label{depth}
The depth $d$ of the potential well is positive.
\end{lemma}
\begin{proof}
Since $q+1\leq2^*$, we have for any $u\in \mathcal{N}$, that
$$a\|\nabla u\|_2^2+b\|\nabla u\|_2^4=\|u\|^{q+1}_{q+1}\leq S^{q+1}\|\nabla u\|_2^{q+1},$$
which implies $\|\nabla u\|_2\geq(\dfrac{a}{S^{q+1}})^{\frac{1}{q-1}}$. Here $S>0$ is the optimal
embedding constant from $H_0^1(\Omega)$ to $L^{q+1}(\Omega)$. By noticing that $q>3$, we have
\begin{eqnarray*}
J(u)&=&\frac{a}{2}\|\nabla u\|_2^2+\frac{b}{4}\|\nabla u\|_2^4-\frac{1}{q+1}\Big(a\|\nabla u\|_2^2+b\|\nabla u\|_2^4\Big)\\
&=&\dfrac{a(q-1)}{2(q+1)}\|\nabla u\|_2^2+\dfrac{b(q-3)}{4(q+1)}\|\nabla u\|_2^4\\
&\geq&\dfrac{a(q-1)}{2(q+1)}(\dfrac{a}{S^{q+1}})^{\frac{2}{q-1}}+\dfrac{b(q-3)}{4(q+1)}(\dfrac{a}{S^{q+1}})^{\frac{4}{q-1}}.
\end{eqnarray*}
Therefore, $d\geq\dfrac{a(q-1)}{2(q+1)}(\dfrac{a}{S^{q+1}})^{\frac{2}{q-1}}+\dfrac{b(q-3)}{4(q+1)}(\dfrac{a}{S^{q+1}})^{\frac{4}{q-1}}>0$.
The proof is complete.
\end{proof}

Now for $\delta>0$, we define some modified functionals and sets as follows:
\begin{eqnarray*}
&& I_{\delta}(u)=\delta(a+b\|\nabla u\|_2^2)\|\nabla u\|_2^2-\|u\|^{q+1}_{q+1},\\
&& \mathcal{N}_{\delta}=\{u\in H_0^1(\Omega)| \ I_{\delta}(u)=0, \ \|\nabla u\|_2\neq0\}.
\end{eqnarray*}
The modified potential wells and their corresponding sets are defined respectively by
\begin{eqnarray*}
&& W_{\delta}=\{u\in H_0^1(\Omega)| \ I_{\delta}(u)>0, \ J(u)<d(\delta)\}\cup\{0\},\\
&& V_{\delta}=\{u\in H_0^1(\Omega)| \ I_{\delta}(u)<0, \ J(u)<d(\delta)\}.
\end{eqnarray*}
Here $d(\delta)=\inf\limits_{u\in \mathcal{N}_{\delta}}J(u)$ is the potential depth of $W_\delta$, which is also positive.

Before investigating the properties of the functionals and sets given above in detail, we present the definition
of weak solutions to Problem \eqref{1.1}.
\begin{definition}\label{de2.1}$\mathrm{\bf{(Weak \ solution)}}$
A function $u=u(x,t)\in L^{\infty}(0, T; H_0^1(\Omega))$ with $u_t\in L^2(0, T; L^2(\Omega))$ is called a weak solution of Problem \eqref{1.1} on $\Omega\times[0,T)$, if
$u(x, 0)=u_0\in H_0^1(\Omega)$ and satisfies
\begin{equation}\label{2.1}
(u_t, \phi)+\Big((a+b\int_{\Omega}|\nabla u|^2\mathrm{d}x)\nabla u, \nabla\phi\Big)=(|u|^{q-1}u, \phi), \ \ \ \ a.\ e. \ t\in(0, T),\\
\end{equation}
for any $\phi\in H_0^1(\Omega)$. Moreover, $u(x,t)$ satisfies
\begin{equation}\label{2.2}
\int_0^t\|u_{\tau}\|_2^2\mathrm{d}\tau+J(u)=J(u_0), \qquad \ a.\ e. \ t\in(0, T).
\end{equation}
\end{definition}

The following lemmas show some basic properties of the functionals and sets defined above,
and will play an important role in the proof of our main results.
Since the proofs are more or less different from the
semi-linear case in one place or another, we also sketch their outlines for the convenience of the readers.

\begin{lemma}\label{le2.1} Let $3<q\leq2^*-1$. Then for any $u\in H_0^1(\Omega)$, $\|\nabla u\|_2\neq0$, we have
\begin{eqnarray*}
\mathrm{(i)} \ \lim_{\lambda\rightarrow0^+}J(\lambda u)=0, \ \lim_{\lambda\rightarrow+\infty}J(\lambda u)=-\infty. \ \ \ \ \ \ \ \ \ \ \ \ \ \ \ \ \ \ \ \ \ \ \ \ \ \ \ \ \ \ \ \ \ \ \ \ \ \ \ \ \ \ \ \ \ \ \ \ \ \
\end{eqnarray*}

$\mathrm{(ii)}$ \ there exists a unique $\lambda^*=\lambda^*(u)>0$ such that $\frac{d}{d\lambda}J(\lambda u)|_{\lambda=\lambda^*}=0$.
$J(\lambda u)$ is increasing on
$0<\lambda\leq\lambda^*$, decreasing on $\lambda^*\leq\lambda<+\infty$ and takes its maximum at $\lambda=\lambda^*$.

$\mathrm{(iii)}$ $I(\lambda u)>0$ on $0<\lambda<\lambda^*$, $I(\lambda u)<0$ on $\lambda^*<\lambda<+\infty$ and $I(\lambda^* u)=0$.
\end{lemma}
\begin{proof}
$\mathrm{(i)}$ From the definition of $J(u)$ we see, for any $\lambda>0$, that
\begin{eqnarray*}
J(\lambda u)=\frac{a\lambda^2}{2}\|\nabla u\|_2^2+\frac{b\lambda^4}{4}\|\nabla u\|_2^4-\frac{\lambda^{q+1}}{q+1}\|u\|_{q+1}^{q+1}.
\end{eqnarray*}
Since $q>3$, it is easy to obtain the results of $\mathrm{(i)}$.

\ \ \ \ $\mathrm{(ii)}$ For any $\lambda>0$, an easy computation shows that
\begin{eqnarray}\label{2.3}
\frac{d}{d\lambda}J(\lambda u)&=&a\lambda\|\nabla u\|_2^2+b\lambda^3\|\nabla u\|_2^4-\lambda^q\|u\|_{q+1}^{q+1}\nonumber\\
&=&\lambda^{q}(a\lambda^{1-q}\|\nabla u\|_2^2+b\lambda^{3-q}\|\nabla u\|_2^4-\|u\|_{q+1}^{q+1}).
\end{eqnarray}
Let
\begin{eqnarray*}
h(\lambda)=a\lambda^{1-q}\|\nabla u\|_2^2+b\lambda^{3-q}\|\nabla u\|_2^4-\|u\|_{q+1}^{q+1}.
\end{eqnarray*}
Recalling the assumption $q>3$ again, we deduce that
\begin{eqnarray}\label{2.4}
h'(\lambda)=a(1-q)\lambda^{-q}\|\nabla u\|_2^2+b(3-q)\lambda^{2-q}\|\nabla u\|_2^4<0,
\end{eqnarray}
and
\begin{eqnarray}\label{2.5}
\lim_{\lambda\rightarrow0^+}h(\lambda)=+\infty, \ \ \lim_{\lambda\rightarrow+\infty}h(\lambda)=-\|u\|_{q+1}^{q+1}<0.
\end{eqnarray}
Therefore, from \eqref{2.4} and \eqref{2.5} it is known that there exists a unique $\lambda^*=\lambda^*(u)>0$ such that $h(\lambda^*)=0$.
Moreover, it follows from \eqref{2.3} that
$\frac{d}{d\lambda}J(\lambda u)|_{\lambda=\lambda^*}=\lambda^{*q}h(\lambda^*)=0$. Since $h(\lambda)>0$ on $(0, \lambda^*)$ and
$h(\lambda)<0$ on $(\lambda^*, +\infty)$, we get that $J(\lambda u)$ is increasing on
$0<\lambda\leq\lambda^*$, decreasing on $\lambda^*\leq\lambda<+\infty$ and takes its maximum at $\lambda=\lambda^*$.

\ \ \ \ $\mathrm{(iii)}$ For any $\lambda>0$, we have
\begin{equation*}
I(\lambda u)=a\|\nabla(\lambda u)\|_2^2+b\|\nabla(\lambda u)\|_2^4-\|\lambda u\|_{q+1}^{q+1}
=\lambda\frac{d}{d\lambda}J(\lambda u).
\end{equation*}
Then the results of $\mathrm{(iii)}$ follow from $\mathrm{(ii)}$ and the above equality. The proof is complete.
\end{proof}

\begin{lemma}\label{le2.2}
Let $3<q\leq2^*-1$, $u\in H_0^1(\Omega)$ and $r(\delta)=\Big(\dfrac{\delta b}{S^{q+1}}\Big)^{\frac{1}{q-3}}$($S$ is the constant given in Lemma \ref{depth}). We have

 \ \ \ \ \ $\mathrm{(i)}$ \ If $0\leq\|\nabla u\|_2\leq r(\delta)$, then $I_{\delta}(u)\geq0.$

 \ \ \ \ \ $\mathrm{(ii)}$ If $I_{\delta}(u)<0$, then $\|\nabla u\|_2>r(\delta).$

 \ \ \ \ $\mathrm{(iii)}$ If $I_{\delta}(u)=0$, then $\|\nabla u\|_2=0$ or $\|\nabla u\|_2\geq r(\delta)$.
\end{lemma}
\begin{proof}
$\mathrm{(i)}$ Since $3<q\leq2^*-1$, from $0\leq\|\nabla u\|_2\leq r(\delta)$ and Sobolev's inequality we obtain
\begin{eqnarray*}
\|u\|^{q+1}_{q+1}\leq S^{q+1}\|\nabla u\|_2^{q+1}= S^{q+1}\|\nabla u\|_2^{q-3}\|\nabla u\|_2^4\leq\delta b\|\nabla u\|_2^4+\delta a\|\nabla u\|_2^2.
\end{eqnarray*}
By the definition of $I_{\delta}(u)$ we see $I_{\delta}(u)\geq0$.

 \ \ \ \ $\mathrm{(ii)}$ \ From $I_{\delta}(u)<0$ and the Sobolev's inequality, we have
\begin{eqnarray*}
\delta b\|\nabla u\|_2^4<\|u\|^{q+1}_{q+1}-\delta a\|\nabla u\|_2^2\leq\|u\|^{q+1}_{q+1}\leq S^{q+1}\|\nabla u\|_2^{q+1},
\end{eqnarray*}
which in turn implies that $\|\nabla u\|_2>\Big(\dfrac{\delta b}{S^{q+1}}\Big)^{\frac{1}{q-3}}=r(\delta)$.

 \ \ \ \ $\mathrm{(iii)}$ \ If $\|\nabla u\|_2=0$, we have $I_{\delta}(u)=0$. If $I_{\delta}(u)=0$ and $\|\nabla u\|_2\neq0$, then by
 the Sobolev's inequality
 \begin{eqnarray*}
\delta b\|\nabla u\|_2^4=\|u\|^{q+1}_{q+1}-\delta a\|\nabla u\|_2^2\leq\|u\|^{q+1}_{q+1}\leq S^{q+1}\|\nabla u\|_2^{q+1},
\end{eqnarray*}
we get $\|\nabla u\|_2\geq\Big(\dfrac{\delta b}{S^{q+1}}\Big)^{\frac{1}{q-3}}=r(\delta)$. The proof is complete.
\end{proof}

\begin{lemma}\label{le2.3}
The function $d(\delta)$ satisfies the following properties:

 \ \ $\mathrm{(i)}$ \ $\lim\limits_{\delta\rightarrow0^+}d(\delta)=0$, $\lim\limits_{\delta\rightarrow+\infty}d(\delta)=-\infty.$

 \ \ $\mathrm{(ii)}$ \ $d(\delta)$ is increasing on $0<\delta\leq1$, decreasing on $\delta\geq1$, and takes its maximum $d=d(1)$

 \ \ \ \ \ \ \ at $\delta=1$.
\end{lemma}
\begin{proof}
$\mathrm{(i)}$ For any $u\in H_0^1(\Omega)$, $\|\nabla u\|_2\neq0$, and for any $\delta>0$, there exists a unique $\lambda=\lambda(\delta)>0$
such that $I_{\delta}(\lambda u)=0$. That is,
\begin{equation}\label{2.6}
\delta (a+b\|\nabla(\lambda u)\|_2^2)\|\nabla(\lambda u)\|_2^2-\|\lambda u\|_q^q=0.
\end{equation}
From \eqref{2.6}, we get
\begin{eqnarray*}
\delta=\frac{\lambda^{q-1}\|u\|^{q+1}_{q+1}}{a\|\nabla u\|_2^2+b\lambda^2\|\nabla u\|_2^4}.
\end{eqnarray*}
It is easily checked from the above expression that $\delta$ is increasing with respect to $\lambda$ on $(0,+\infty)$, which implies that the inverse function
$\lambda(\delta)$ is also increasing on $\delta>0$. Furthermore, we can deduce from \eqref{2.6} that $\lim\limits_{\delta\rightarrow0}\lambda(\delta)=0$ and ${\lim\limits_{\delta\rightarrow+\infty}}
\lambda(\delta)=+\infty$. Since $\lambda u\in \mathcal{N}_{\delta}$, $d(\delta)\leq J(\lambda u)$. It follows that
\begin{eqnarray*}
0\leq\lim\limits_{\delta\rightarrow0}d(\delta)\leq\lim\limits_{\delta\rightarrow0}J(\lambda u)=\lim\limits_{\lambda\rightarrow0}J(\lambda u)=0,
\end{eqnarray*}
i.e. $\lim\limits_{\delta\rightarrow0}d(\delta)=0$. On the other hand,
\begin{eqnarray*}
\lim\limits_{\delta\rightarrow+\infty}d(\delta)\leq\lim\limits_{\delta\rightarrow+\infty}J(\lambda u)=\lim\limits_{\lambda\rightarrow+\infty}J(\lambda u)=-\infty,
\end{eqnarray*}
that is $\lim\limits_{\delta\rightarrow+\infty}d(\delta)=-\infty.$

 \ \ \ \ \ \ $\mathrm{(ii)}$ Clearly, we only need to prove that for any $0<\delta'<\delta''<1$ or $\delta'>\delta''>1$ and any $u\in \mathcal{N}_{\delta''}$
 there exist a $v\in \mathcal{N}_{\delta'}$ and a constant $\varepsilon(\delta', \delta'')>0$ such that $J(u)-J(v)>\varepsilon(\delta', \delta'')$.
In fact, for any $u\in \mathcal{N}_{\delta''}$, we have $\lambda(\delta'')=1$ and $\|\nabla u\|_2\geq r(\delta'')$. Take $v=\lambda(\delta')u$, then $v\in \mathcal{N}_{\delta'}$. Let $g(\lambda)=
J(\lambda(\delta)u)$, then
\begin{eqnarray*}
\frac{d}{d\lambda}g(\lambda)&=&\frac{1}{\lambda}[a(1-\delta)\|\nabla(\lambda u)\|_2^2+b(1-\delta)\|\nabla(\lambda u)\|_2^4+I_{\delta}(\lambda u)]\\
&=&a(1-\delta)\lambda\|\nabla u\|_2^2+b(1-\delta)\lambda^{3}\|\nabla u\|_2^4.
\end{eqnarray*}
If $0<\delta'<\delta''<1$, since $\lambda(\delta)$ is increasing in $\delta$ and $\lambda(\delta'')=1$, then
\begin{eqnarray*}
J(u)-J(v)&=&g(1)-g(\lambda(\delta'))=\int^1_{\lambda(\delta')}\frac{d}{d\lambda}g(\lambda)\mathrm{d}\lambda\\
&=&\int^1_{\lambda(\delta')}[a(1-\delta)\lambda\|\nabla u\|_2^2+(1-\delta)\lambda^{3}\|\nabla u\|_2^4]\mathrm{d}\lambda\\
&\geq&[a\lambda(\delta')(1-\delta'')r^2(\delta'')+b\lambda^3(\delta')(1-\delta'')r^4(\delta'')](1-\lambda(\delta'))\\
&=&\varepsilon(\delta', \delta'')>0.
\end{eqnarray*}
If $\delta'>\delta''>1$, then
\begin{eqnarray*}
J(u)-J(v)&=&g(1)-g(\lambda(\delta'))=\int^1_{\lambda(\delta')}\frac{d}{d\lambda}g(\lambda)\mathrm{d}\lambda\\
&=&\int_1^{\lambda(\delta')}[a(\delta-1)\lambda\|\nabla u\|_2^2+b(\delta-1)\lambda^{3}\|\nabla u\|_2^4]\mathrm{d}\lambda\\
&\geq&[a\lambda(\delta'')(\delta''-1)r^2(\delta'')+b\lambda^3(\delta'')(\delta''-1)r^4(\delta'')](\lambda(\delta')-1)\\
&=&\varepsilon(\delta', \delta'')>0.
\end{eqnarray*}
Furthermore, since $d(\delta)$ is continuous with respect to $\delta$ and from the results obtained in $\mathrm{(i)}$,
we see that there exists a $\widetilde{\delta}>1$ such that $d(\widetilde{\delta})=0$. The proof is complete.
\end{proof}

\begin{lemma}\label{le2.4}
Assume $u\in H_0^1(\Omega)$, $0<J(u)<d$, and $\delta_1<1<\delta_2$ are the two roots of the equation $d(\delta)=J(u)$. Then the sign of $I_{\delta}(u)$
does not change for $\delta_1<\delta<\delta_2$.
\end{lemma}
\begin{proof}
First $J(u)>0$ implies $\|\nabla u\|_2\neq0$. If the sign of $I_{\delta}(u)$ changes for $\delta_1<\delta<\delta_2$, then there exists a $
\bar{\delta}\in(\delta_1, \delta_2)$ such that $I_{\bar{\delta}}(u)=0$. Thus by the definition of $d(\delta)$ we have $J(u)\geq d({\bar{\delta}})$,
which is contradictive with $J(u)=d(\delta_1)=d(\delta_2)<d(\bar{\delta})$.
\end{proof}

\begin{lemma}\label{le2.5}
Assume that $u(x, t)$ is a weak solution of Problem \eqref{1.1} with $0<J(u_0)<d$
and $T$ is the maximal existence time. Let $\delta_1<1<\delta_2$ be the two roots of the equation $d(\delta)=J(u_0)$.

 \ \ $\mathrm{(i)}$ If $I(u_0)>0$, then $u(x, t)\in W_{\delta}$ for $\delta_1<\delta<\delta_2$ and $0<t<T$.

 \ \ $\mathrm{(ii)}$ If $I(u_0)<0$, then $u(x, t)\in V_{\delta}$ for $\delta_1<\delta<\delta_2$ and $0<t<T$.
\end{lemma}
\begin{proof}
$\mathrm{(i)}$ For $ 0<J(u_0)=d(\delta_1)=d(\delta_2)<d$, $I(u_0)>0$, from Lemma \ref{le2.4} we know $u_0\in W_{\delta}$ for all $\delta_1<\delta<\delta_2$.
Next we will prove $u(t)\in W_{\delta}$ for all $\delta_1<\delta<\delta_2$ and $0<t<T$. Otherwise, there exists
a $t_0\in(0, T)$ and a $\delta_0\in(\delta_1, \delta_2)$ such that $u(t_0)\in\partial W_{\delta_0}$.
Noticing that $0$ is an interior point of $W_\delta$ for any $\delta_1<\delta<\delta_2$, we thus have
\begin{eqnarray*}
I_{\delta_0}(u(t_0))=0, \ \|\nabla u(t_0)\|_2\neq0, \ \ \text{or} \ \ J(u(t_0))=d(\delta_0).
\end{eqnarray*}
As $J(u(t_0))<d(\delta_0)$ by \eqref{2.2}, we thus have $I_{\delta_0}(u(t_0))=0$ and $\|\nabla u(t_0)\|_2\neq0$, which, by the definition of $d(\delta_0)$, implies that
$J(u(t_0))\geq d(\delta_0)$, a contradiction to \eqref{2.2}.

 \ \ \ \ $\mathrm{(ii)}$ Similarly, we have $u_0\in V_{\delta}$ for all $\delta_1<\delta<\delta_2$. Next we will show that $u(t)\in V_{\delta}$ for all $\delta_1<\delta<\delta_2$ and $0<t<T$. If not, there exist
a $t_0\in(0, T)$ and a $\delta_0\in(\delta_1, \delta_2)$ such that $u(t_0)\in\partial V_{\delta_0}$, namely
\begin{eqnarray*}
I_{\delta_0}(u(t_0))=0, \ \ \text{or} \ \ J(u(t_0))=d(\delta_0).
\end{eqnarray*}
By \eqref{2.2}, we can see that $J(u(t_0))\neq d(\delta_0)$, then $I_{\delta_0}(u(t_0))=0$. We assume that $t_0$ is the first time such that
$I_{\delta_0}(u(t))=0$, then $I_{\delta_0}(u(t))<0$ for $0\leq t<t_0$. By Lemma \ref{le2.2}$\mathrm{(ii)}$ we have $\|\nabla u\|_2>r(\delta_0)$
for $0\leq t<t_0$. Hence $\|\nabla u(t_0)\|_2\geq r(\delta_0)$, which together with $I_{\delta_0}(u(t_0))=0$ implies that $u(t_0)\in \mathcal{N}_{\delta_0}$.
By the definition of $d(\delta_0)$, we again obtain $J(u(t_0))\geq d(\delta_0)$, a contradiction to \eqref{2.2}. The proof is complete.
\end{proof}

\par
\section{$\mathrm{The \ case \ J(u_0)<d}$.}
\setcounter{equation}{0}

In this section we consider the behaviors of the solution of Problem \eqref{1.1} under the condition $J(u_0)<d$
and give the threshold result for the solutions to exist globally or to blow up in finite time.
Before stating and proving our main results, we first derive some basic properties of the nonlocal Laplacian $-(a+b\|\nabla u\|_2^2)\triangle u$ in \eqref{1.1},
which are also of independent interest.

Consider the following functional:
\begin{eqnarray*}
E(u)=\Big(\dfrac{a}{2}+\dfrac{b}{4}\int_{\Omega}|\nabla u|^2\mathrm{d}x\Big)\int_{\Omega}|\nabla u|^2\mathrm{d}x, \ \ \ u\in H_0^1(\Omega).
\end{eqnarray*}
It is easy to see that $E\in C^1(H_0^1(\Omega), R)$,
and the nonlocal operator is the Fr\'{e}chet derivative operator of $E$ in the weak sense.
Denote $L=E':H_0^1(\Omega)\rightarrow H^{-1}(\Omega)$, then
\begin{eqnarray*}
\langle L(u),v\rangle=(a+b\|\nabla u\|_2^2)\int_{\Omega}\nabla u\nabla v\mathrm{d}x,\quad\forall\ u,\,v\in H_0^1(\Omega).
\end{eqnarray*}
Here $\langle,\rangle$ denotes the pairing between $H^{-1}(\Omega)$ and $H_0^1(\Omega)$.
For the nonlocal Laplacian $L$, we have the following important properties.

\begin{lemma}\label{lem-nonlocal}
$\mathrm{(i)}$ $L: H_0^1(\Omega)\rightarrow H^{-1}(\Omega)$ is a continuous, bounded and strongly monotone operator.\\
$\mathrm{(ii)}$ $L$ is a mapping of type $(S_{+})$, i.e. if $u_n\rightharpoonup u$ weakly in $H_0^1(\Omega)$ and
$\overline{\lim\limits_{n\rightarrow\infty}}\langle L(u_n),u_n-u\rangle\leq 0$, then $u_n\rightarrow u$ strongly in $H_0^1(\Omega)$.
\end{lemma}

\begin{proof}
$\mathrm{(i)}$ We say that an operator $L: H_0^1(\Omega)\rightarrow H^{-1}(\Omega)$ is strongly monotone if and only if there exists a positive constant $c$
such that
$$\langle L(u)-L(v), u-v\rangle\geq c\|u-v\|^2_{H_0^1(\Omega)},\quad\forall\ u,v\in H_0^1(\Omega).$$
It is obvious that $L$ is continuous and bounded. For any  $u, v\in H_0^1(\Omega)$, by using Cauchy-Schwarz inequality we have
\begin{eqnarray*}\label{L.1}
&&\langle L(u)-L(v), u-v\rangle\\
&=&\int_{\Omega}((a+b\|\nabla u\|_2^2)\nabla u-(a+b\|\nabla v\|_2^2)\nabla v)(\nabla u-\nabla v)\mathrm{d}x\\
&=&a\|\nabla(u-v)\|_2^2+b\int_{\Omega}(\|\nabla u\|_2^2\nabla u-\|\nabla v\|_2^2\nabla v)(\nabla u-\nabla v)\mathrm{d}x\\
&=&a\|u-v\|_{H_0^1(\Omega)}^2+b\Big(\|\nabla u\|_2^4-\|\nabla u\|_2^2\int_{\Omega}\nabla u\nabla v\mathrm{d}x-\|\nabla v\|_2^2\int_{\Omega}\nabla u\nabla v\mathrm{d}x+\|\nabla v\|_2^4\Big)\\
&\geq&a\|u-v\|_{H_0^1(\Omega)}^2+b\Big(\|\nabla u\|_2^4-\|\nabla u\|_2^2\frac{\|\nabla u\|_2^2\|+\nabla v\|_2^2}{2}-\|\nabla v\|_2^2\frac{\|\nabla u\|_2^2+\|\nabla v\|_2^2}{2}+\|\nabla v\|_2^4\Big)\\
&=&a\|u-v\|_{H_0^1(\Omega)}^2+\frac{b}{2}(\|\nabla u\|_2^2-\|\nabla v\|_2^2)^2\\
&\geq&a\|u-v\|_{H_0^1(\Omega)}^2.
\end{eqnarray*}
Therefore, the strongly monotonicity of $L$ is proved.

$\mathrm{(ii)}$ If $u_n\rightharpoonup u$ weakly in $H_0^1(\Omega)$ and $\overline{\lim\limits_{n\rightarrow\infty}}
\langle L(u_n),u_n-u\rangle\leq 0$, then we have
\begin{equation*}
\overline{\lim\limits_{n\rightarrow\infty}}\langle L(u_n)-L(u),u_n-u\rangle\leq0,
\end{equation*}
which, together with the strongly monotonicity of $L$, implies that $u_n\rightarrow u$ strongly in $H_0^1(\Omega)$.
Hence $L$ is an $S_{+}$ operator. The proof is complete.
\end{proof}

\begin{theorem}\label{th3.1}(Global existence for $J(u_0)<d$.)
Assume $a, b >0$, $3<q\leq2^*-1$ and $u_0\in H_0^1(\Omega)$. If
$J(u_0)<d$ and $I(u_0)>0$, then Problem \eqref{1.1} admits a global weak solution $u\in L^{\infty}(0,\infty; H_0^1(\Omega))$
with $u_t\in L^2(0,\infty; L^2(\Omega))$ and $u(t)\in W$ for $0\leq t<\infty$.
Moreover, $\|u\|_2^2\leq\|u_0\|_2^2e^{-2a\lambda_1(1-\delta_1)t}$,
where $\lambda_1>0$ is the first eigenvalue of $-\Delta$ in $\Omega$ under homogeneous Dirichlet boundary condition.
In addition, the weak solution is unique if it is bounded.
\end{theorem}
\begin{proof}
We will divide the proof into three steps for the convenience of the readers.\\
{\bf Step 1. Global existence.} Global existence of weak solutions will be proved by combining Galerkin's approximation with a priori estimates.
Let $\{\phi_j(x)\}$ be a system of orthogonal basis of $H_0^1(\Omega)$ and construct the approximate solutions $u^m(x, t)$ of Problem \eqref{1.1}
\begin{eqnarray*}
u^m(x, t)=\sum_{j=1}^m a^m_j(t)\phi_j(x), \ \ m=1, \ 2, \ \cdots,
\end{eqnarray*}
satisfying
\begin{equation}\label{3.1}
(u^m_t, \phi_j)+a(\nabla u^m, \nabla\phi_j)+b\|\nabla u^m\|_2^2(\nabla u^m, \nabla\phi_j)=(|u^m|^{q-1}u^m, {\phi_j}),\ j=1,2,\cdots,
\end{equation}
\begin{equation}\label{3.2}
u^m(x,0)=\sum_{j=1}^m b^m_j\phi_j(x)\rightarrow u_0(x) \ \ in \ \ H_0^1(\Omega). \ \ \ \ \ \ \ \ \ \ \ \ \ \ \ \ \ \ \ \ \ \ \ \ \ \
\end{equation}
Multiplying \eqref{3.1} by $\dfrac{d}{dt}a^m_j(t)$, summing for $j$ from $1$ to $m$, and integrating with respect to $t$ from $0$ to $t$, we obtain
\begin{eqnarray}\label{3.3}
\int_0^t\|u^m_{\tau}\|_2^2\mathrm{d}\tau+J(u^m)=J(u^m(0)), \ \ \ 0\leq t<\infty.
\end{eqnarray}
Due to the convergence of $u^m(x, 0)\rightarrow u_0(x)$ in $H_0^1(\Omega)$, we have
\begin{eqnarray*}
J(u^m(x, 0))\rightarrow J(u_0(x))<d\quad \text{and} \quad I(u^m(x, 0))\rightarrow I(u_0(x))>0.
\end{eqnarray*}
Then for sufficiently large $m$ and for any $0\leq t<\infty$, we obtain
\begin{eqnarray}\label{3.4}
\int_0^t\|u^m_{\tau}\|_2^2\mathrm{d}\tau+J(u^m)=J(u^m(0))<d \quad \text{and} \quad I(u^m(x, 0))>0.
\end{eqnarray}

Similarly to the proof of Lemma \ref{le2.5} we can show from \eqref{3.4} that $u^m(x,t)\in W$ for sufficiently large $m$ and $0\leq t<\infty$.
Thus $I(u^m(x,t))>0$ for all $t\geq0$. Then from the following equality
\begin{eqnarray*}
J(u^m)=\frac{a(q-1)}{2(q+1)}\|\nabla u^m\|_2^2+\frac{b(q-3)}{4(q+1)}\|\nabla u^m\|^4_2+\frac{1}{q+1}I(u^m),
\end{eqnarray*}
and \eqref{3.4} we obtain
\begin{eqnarray}\label{3.5}
\int_0^t\|u^m_{\tau}\|_2^2\mathrm{d}\tau+\frac{a(q-1)}{2(q+1)}\|\nabla u^m\|_2^2+\frac{b(q-3)}{4(q+1)}\|\nabla u^m\|^4_2<d,
\end{eqnarray}
for sufficiently large $m$ and for any $0\leq t<\infty$, which then yields
\begin{equation}\label{e1}
\|u\|^2_{H_0^1(\Omega)}\leq \dfrac{2d(q+1)}{a(q-1)},\quad 0\leq t<\infty,
\end{equation}
\begin{equation}\label{e2}
\int_0^t\|u^m_{\tau}\|_2^2\mathrm{d}\tau<d,\quad 0\leq t<\infty,
\end{equation}
\begin{equation}\label{e3}
\||u^m|^{q-1}u^m\|_{\frac{q+1}{q}}=\|u^m\|^q_{q+1}\leq S^q\|u\|^q_{H_0^1(\Omega)}<S^q\Big(\dfrac{2d(q+1)}{a(q-1)}\Big)^{\frac{q}{2}},\quad 0\leq t<\infty.
\end{equation}
Therefore, by the diagonal method there exist a $u$ and a subsequence of $\{u^m\}$ (still denoted by $\{u^m\}$) such that for each $T>0$, as $m\rightarrow\infty$,
\begin{eqnarray}\label{3.6}
\begin{cases}
u^m_t\rightharpoonup u_t, \ & weakly \ in \ L^2(0, T; L^2(\Omega)),\\
u^m\rightharpoonup u, \ & weakly \ in \ L^{2}(0, T; H_0^1(\Omega)),\\
u^m\rightarrow u, & \ strongly \ in \ L^2(\Omega\times(0,T))\ and\ a.e.\ in \ \Omega\times(0,T),\\
|u^m|^{q-1}u^m\rightharpoonup |u|^{q-1}u, & \ weakly \ in \ L^{\frac{q+1}{q}}(\Omega\times(0, T)).
\end{cases}
\end{eqnarray}
Hence for $j$ fixed and letting $m\rightarrow\infty$ in \eqref{3.1}, one has
\begin{eqnarray*}
(u_t, \phi_j)+a(\nabla u, \nabla\phi_j)+b\lim_{m\rightarrow\infty}\|\nabla u^m\|_2^2(\nabla u, \nabla\phi_j)=(|u|^{q-1}u, {\phi_j}).
\end{eqnarray*}
Then for any $\varphi\in H_0^1(\Omega)$,
\begin{eqnarray}\label{3.71}
(u_t, \varphi)+a(\nabla u, \nabla\varphi)+b\lim_{m\rightarrow\infty}\|\nabla u^m\|_2^2(\nabla u, \nabla\varphi)=(|u|^{q-1}u, \varphi).
\end{eqnarray}
Choosing $\varphi=u$ in \eqref{3.71}, we have
\begin{eqnarray}\label{3.81}
(u_t, u)+a(\nabla u, \nabla u)+b\lim_{m\rightarrow\infty}\|\nabla u^m\|_2^2(\nabla u, \nabla u)=(|u|^{q-1}u, u).
\end{eqnarray}
On the other hand, choosing $\phi_j=u^m$ in \eqref{3.1}, we get
\begin{eqnarray}\label{3.91}
(u^m_t, u^m)+a(\nabla u^m, \nabla u^m)+b\|\nabla u^m\|_2^2(\nabla u^m, \nabla u^m)=(|u^m|^{q-1}u^m, u^m)
\end{eqnarray}
Using the convergence in \eqref{3.6}, letting $m\rightarrow\infty$ in \eqref{3.91} and comparing it with \eqref{3.81}, we obtain
\begin{equation}
\lim_{m\rightarrow\infty}\|\nabla u^m\|_2=\|\nabla u\|_2.
\end{equation}
Therefore, \eqref{3.71} shows that for any $\varphi\in H_0^1(\Omega)$
\begin{eqnarray}\label{3.11}
(u_t, \varphi)+a(\nabla u, \nabla\varphi)+b\|\nabla u\|_2^2(\nabla u, \nabla\varphi)=(|u|^{q-1}u, \varphi).
\end{eqnarray}
Besides, due to $u^m(x, 0)\rightarrow u_0(x)$ strongly in $H_0^1(\Omega)$, we have $u(x, 0)=u_0(x)$.
To prove \eqref{2.2} we first assume that $u(x,t)$ is smooth enough such that $u_t\in L^2(0,T;H_0^1(\Omega))$.
Choosing $\phi=u_t$ as a test function and integrating \eqref{2.1} over $[0,t]$ one sees that \eqref{2.2} is true.
By the density of $L^2(0,T;H_0^1(\Omega))$ in $L^2(\Omega\times(0,T))$ it is known that \eqref{2.2} also holds for weak solutions
of \eqref{1.1}. Therefore $u$ is a global weak solution of Problem \eqref{1.1}.

{\bf Step 2. Decay rate.} Taking $\phi=u$ in \eqref{2.1}, we get
\begin{eqnarray*}
\frac{1}{2}\frac{d}{dt}\|u\|_2^2=(u_t, u)=-a\|\nabla u\|_2^2-b\|\nabla u\|_2^4+\|u\|_{q+1}^{q+1}=-I(u).
\end{eqnarray*}
From Lemma \ref{le2.4} it follows that $u(x, t)\in W_{\delta}$ for $\delta_1<\delta<\delta_2$ and $0<t<\infty$ under the condition
$J(u_0)<d$ and $I(u_0)>0$. Thus we have $I_{\delta_1}(u)\geq0$ for $0<t<\infty$. Therefore,
\begin{eqnarray*}
\frac{1}{2}\frac{d}{dt}\|u\|_2^2=-I(u)=a(\delta_1-1)\|\nabla u\|_2^2+b(\delta_1-1)\|\nabla u\|_2^4-I_{\delta_1}(u)\leq a\lambda_1(\delta_1-1)\|u\|_2^2,
\end{eqnarray*}
Consequently,
\begin{eqnarray*}
\|u\|_2^2\leq\|u_0\|_2^2e^{-2a\lambda_1(1-\delta_1)t}.
\end{eqnarray*}

{\bf Step 3. Uniqueness of bounded solution.} To prove the uniqueness of bounded weak solution,
we assume that both $u$ and $v$ are bounded weak solutions of Problem \eqref{1.1}.
Then, for any $\varphi\in H_0^1(\Omega)$, we have
\begin{eqnarray*}
&&(u_t, \varphi)+a(\nabla u, \nabla\varphi)+b\|\nabla u\|_2^2(\nabla u, \nabla\varphi)=(|u|^{q-1}u, \varphi),\\
&&(v_t, \varphi)+a(\nabla v, \nabla\varphi)+b\|\nabla v\|_2^2(\nabla v, \nabla\varphi)=(|v|^{q-1}v, \varphi).
\end{eqnarray*}
Subtracting the above two equalities, taking $\varphi=u-v\in H_0^1(\Omega)$ and integrating over $(0,t)$ for any $t>0$, we obtain
\begin{eqnarray*}
&&\int_0^t\int_{\Omega}(u-v)_t(u-v)+a|\nabla(u-v)|^2+(b\|\nabla u\|_2^2\nabla u-b\|\nabla v\|_2^2\nabla v)\nabla(u-v)\mathrm{d}x\mathrm{d}t\\
&=&\int_0^t\int_{\Omega}(|u|^{q-1}u-|v|^{q-1}v)(u-v)\mathrm{d}x\mathrm{d}t.
\end{eqnarray*}
Since $(u-v)(x,0)=0$ and $q>3$, we obtain, with the help of Lemma \ref{lem-nonlocal} and the boundedness of $u$ and $v$, that
\begin{eqnarray*}
\int_{\Omega}(u-v)^2(x, t)\mathrm{d}x\leq C\int_0^t\int_{\Omega}(u-v)^2(x, t)\mathrm{d}x\mathrm{d}t,
\end{eqnarray*}
where $C>0$ is a constant depending only on $q$ and the bound of $u,v$.
It then follows from Gronwall's inequality that
\begin{eqnarray*}
\int_{\Omega}w^2(x, t)\mathrm{d}x=0.
\end{eqnarray*}
Thus $w=0$ a.e. in $\Omega\times(0, \infty)$ and the whole proof is complete.
\end{proof}

\begin{theorem}\label{th3.2}(Blow-up for $J(u_0)<d$.)
Assume $a, b>0$, $3<q\leq2^*-1$ and let $u$ be a weak solution of Problem \eqref{1.1} with $u_0\in H_0^1(\Omega)$. If $J(u_0)<d$
and $I(u_0)<0$, then there exists a finite time $T$ such that $u$ blows up at $T$ in the sense that
\begin{eqnarray*}
\lim_{t\rightarrow T}\int_0^t\|u\|_2^2\mathrm{d}\tau=+\infty.
\end{eqnarray*}
\end{theorem}
\begin{proof}
Let $u$ be a weak solution of Problem \eqref{1.1} with $J(u_0)<d$, $I(u_0)<0$. We define
\begin{eqnarray*}
M(t)=\int_0^t\|u\|_2^2\mathrm{d}\tau,
\end{eqnarray*}
then
\begin{eqnarray}\label{3.6}
M'(t)=\|u\|_2^2,
\end{eqnarray}
and
\begin{eqnarray}\label{3.7}
M''(t)=2(u_t, u)=-2(a\|\nabla u\|_2^2+b\|\nabla u\|_2^4-\|u\|^{q+1}_{q+1})=-2I(u).
\end{eqnarray}
On the other hand,
\begin{eqnarray}\label{3.8}
J(u)=\frac{a(q-1)}{2(q+1)}\|\nabla u\|_2^2+\frac{b(q-3)}{4(q+1)}\|\nabla u\|_2^4+\frac{1}{q+1}I(u).
\end{eqnarray}
By \eqref{2.1}, \eqref{3.7} and \eqref{3.8}, we can get
\begin{eqnarray*}
M''(t)&=&a(q-1)\|\nabla u\|_2^2+\frac{b(q-3)}{2}\|\nabla u\|^4_2-2(q+1)J(u)\\
&\geq&a(q-1)\|\nabla u\|_2^2+\frac{b(q-3)}{2}\|\nabla u\|^4_2+2(q+1)\int_0^t\|u_{\tau}\|_2^2\mathrm{d}\tau-2(q+1)J(u_0)\\
&\geq&a(q-1)\lambda_1M'(t)+2(q+1)\int_0^t\|u_{\tau}\|_2^2\mathrm{d}\tau-2(q+1)J(u_0).
\end{eqnarray*}
Note that
\begin{eqnarray*}
(M'(t))^2=4\bigg(\int_0^t\int_{\Omega} u_{\tau}\cdot u\mathrm{d}x\mathrm{d}\tau\bigg)^2+2\|u_0\|_2^2M'(t)-\|u_0\|_2^4.
\end{eqnarray*}
Hence, we have
\begin{eqnarray*}
M''(t)M(t)-\frac{q+1}{2}M'(t)^2&\geq&2(q+1)\int_0^t\|u_{\tau}\|_2^2\mathrm{d}\tau\int_0^t\|u\|_2^2\mathrm{d}\tau
-2(q+1)J(u_0)M(t)\\
&& + \ a(q-1)\lambda_1M'(t)M(t)-2(q+1)\bigg(\int_0^t\int_{\Omega}u_{\tau}u\mathrm{d}x\mathrm{d}\tau\bigg)^2\\
&& - \ (q+1)\|u_0\|^2_2M'(t)+\frac{q+1}{2}\|u_0\|_2^4.
\end{eqnarray*}
Applying Cauchy-Schwartz inequality to the fourth term of the right-hand side of the above inequality
and dropping the last positive one, we get
\begin{eqnarray}\label{3.9}
&&M''(t)M(t)-\frac{q+1}{2}M'(t)^2\nonumber\\
&\geq&a(q-1)\lambda_1M'(t)M(t)-2(q+1)J(u_0)M(t)-(q+1)\|u_0\|^2_2M'(t).
\end{eqnarray}
Next, we discuss the following two cases.\\
$\mathrm{(i)}$ If $J(u_0)\leq 0$, then \eqref{3.9} implies
\begin{eqnarray*}
M''(t)M(t)-\frac{q+1}{2}M'(t)^2\geq a(q-1)\lambda_1M'(t)M(t)-(q+1)\|u_0\|^2_2M'(t).
\end{eqnarray*}
Now we will prove that $I(u)<0$ for all $t>0$. Otherwise, there must be a $t_0>0$ such that $I(u(t_0))=0$ and
$I(u(t))<0$ for $0\leq t<t_0$. From Lemma \ref{le2.2}(ii), $\|\nabla u\|_2>r(1)$ for $0\leq t<t_0$, and
$\|\nabla u(t_0)\|_2\geq r(1)$. Hence, by the definition of $d$, we have $J(u(t_0))\geq d$, which contradicts \eqref{2.2}.
Then from \eqref{3.7}, we can get $M''(t)>0$, for $t\geq 0$. Since $M'(0)\geq0$, there exists a $t_0\geq0$ such that $M'(t_0)>0$.
Thus we have
\begin{eqnarray*}
M(t)\geq M'(t_0)(t-t_0).
\end{eqnarray*}
Therefore, for sufficiently large t, we have
\begin{eqnarray}\label{3.10}
a(q-1)\lambda_1M(t)> (q+1)\|u_{0}\|_2^2.
\end{eqnarray}
Consequently,
\begin{eqnarray*}
M''(t)M(t)-\frac{q+1}{2}M'(t)^2>0.
\end{eqnarray*}
(ii) If $0<J(u_0)<d$, then by Lemma \ref{le2.5} we have $u(t)\in V_{\delta}$ for $t\geq0$ and $\delta_1<\delta<\delta_2$,
where $\delta_1<1<\delta_2$ are the two roots of $d(\delta)=J(u_0)$. Hence $I_{\delta_2}(u)\leq 0$ and $\|\nabla u\|_2^2\geq r(\delta_2)$ for
$t\geq 0$. By \eqref{3.7} we have for $t\geq0$
\begin{eqnarray*}
M''(t)&=&-2I(u)=2a(\delta_2-1)\|\nabla u\|_2^2+2b(\delta_2-1)\|\nabla u\|_2^4-2I_{\delta_2}(u)\\
&\geq& 2a(\delta_2-1)r^2(\delta_2).
\end{eqnarray*}
It follows then for all $t\geq 0$ that
\begin{eqnarray*}
&&M'(t)\geq 2a(\delta_2-1)r^2(\delta_2)t, \\
&&M(t)\geq a(\delta_2-1)r^2(\delta_2)t^2.
\end{eqnarray*}
Therefore, for sufficiently large $t$, we have
\begin{eqnarray*}
&&\frac{a(q-1)\lambda_1}{2}M(t)>(q+1)\|u_0\|_2^2,\\
&&\frac{a(q-1)\lambda_1}{2}M'(t)>2(q+1)J(u_0).
\end{eqnarray*}
Consequently, from \eqref{3.9}, we obtain
\begin{eqnarray*}
M''(t)M(t)-\frac{q+1}{2}M'(t)^2>0.
\end{eqnarray*}
The remainder of the proof follows from the standard concavity arguments as those in \cite{Levine1973,Payne1975}
and the details are therefore omitted. The proof is complete.
\end{proof}

\par
\section{$\mathrm{The \ case \ J(u_0)=d}$.}
\setcounter{equation}{0}

For the critical case $J(u_0)=d$, we have also obtained the following threshold results.
\begin{theorem}\label{th4.1}
Assume $a, b>0$, $3<q\leq2^*-1$, $u_0\in H_0^1(\Omega)$. If $J(u_0)=d$ and $I(u_0)\geq0$, then Problem \eqref{1.1} admits a global weak
solution $u\in L^{\infty}(0,\infty; H_0^1(\Omega))$ with $u_t\in L^2(0,\infty; L^2(\Omega))$ and $I(u)\geq 0$. Moreover, if $I(u)>0$, the solution
does not vanish and there exist positive constants $C_1$ and $C_2$ such that $\|u\|_2^2\leq C_1e^{-C_2t}$.
If not, then there exists a solution that vanishes in finite time.
\end{theorem}
\begin{proof}
Let $\lambda_k=1-\frac{1}{k}$ $k=1, 2,\cdots$. Consider the following initial and boundary value problem
\begin{eqnarray}\label{4.1}
\begin{cases}
u_t-(a+b\|\nabla u\|_2^2)\triangle u=|u|^{q-1}u, \ & (x,t)\in \Omega\times(0,T),\\
u=0, \ & (x,t)\in \partial \Omega\times(0,T),\\
u(x, 0)=\lambda_k u_0, \ & x\in\Omega.
\end{cases}
\end{eqnarray}
Since $I(u_0)\geq0$, $q>3$, we can deduce that there exists a unique $\lambda^*=\lambda^*(u_0)\geq1$ such that $I(\lambda^* u_0)=0$. And then from $\lambda_k<1\leq\lambda^*$, we get $I(u_0^k)=I(\lambda_k u_0)>0$
and $J(u_0^k)=J(\lambda_k u_0)<J(u_0)=d$. In view of Theorem \ref{th3.1}, it follows that for each $k$ Problem \eqref{4.1}
admits a global weak solution $u^k$ satisfying $u^k\in L^{\infty}(0,\infty; H_0^1(\Omega))$, $u^k_t\in L^2(0,\infty; L^2(\Omega))$,
$u^k\in W$ for $0<t<\infty$ and $\int_0^t\|u^k_{\tau}\|_2^2d\tau+J(u^k)=J(u_0^k)<d$.
Applying the arguments similar to those in Theorem \ref{th3.1} we see that there exist a subsequence of $\{u^k\}$ and a function $u$, such
that $u$ is the weak solution of Problem \eqref{1.1} with $I(u)\geq0$ and $J(u)\leq d$
for $0<t<\infty$.

Next, Let us consider the asymptotic behavior. First, suppose that $I(u)>0$ for $0<t<\infty$, then $u(x, t)$ does not vanish in finite time.
Taking $\phi=u$ in \eqref{2.1}, we have
\begin{eqnarray*}
\frac{1}{2}\frac{d}{dt}\|u\|_2^2=\int_{\Omega}u_tu\mathrm{d}x=-I(u)<0,
\end{eqnarray*}
which implies that $u_t\not\equiv0$. Therefore, by \eqref{2.2} there exists a $t_0>0$ such that
\begin{eqnarray*}
0<J(u(t_0))=d-\int_0^{t_0}\|u_{\tau}\|_2^2\mathrm{d}\tau=d_1<d.
\end{eqnarray*}
Choosing $t=t_0$ as the initial time and  by Lemma \ref{le2.5}, we get that $u\in W_{\delta}$
for $\delta_1<\delta<\delta_2$ and $t>t_0$, where $\delta_1<1<\delta_2$ are the two roots of $d(\delta)=d_1$.
Hence, $I_{\delta_1}(u)\geq0$ for $t>t_0$ and
\begin{eqnarray*}
\frac{1}{2}\frac{d}{dt}\|u\|_2^2=-I(u)=a(\delta_1-1)\|\nabla u\|_2^2+b(\delta_1-1)\|\nabla u\|_2^4-I_{\delta_1}(u)\leq a\lambda_1(\delta_1-1)\|u\|_2^2.
\end{eqnarray*}
Therefore,
\begin{eqnarray*}
\|u\|_2^2\leq\|u(t_0)\|_2^2e^{-2a\lambda_1(1-\delta_1)(t-t_0)}.
\end{eqnarray*}
That is $\|u\|_2^2\leq C_1e^{-C_2t}$.

Next, suppose $I(u)>0$ for $0<t<t_0$ and $I(u(x, t_0))=0$. Obviously, $u_t\not\equiv0$ for $0<t<t_0$
and $\int_0^{t_0}\|u_{\tau}\|_2^2\mathrm{d}\tau>0$. Recalling \eqref{2.2}, we have
\begin{eqnarray*}
J(u(t_0))=d-\int_0^{t_0}\|u_{\tau}\|_2^2\mathrm{d}\tau=d_1<d.
\end{eqnarray*}
By the definition of $d$, we know $\|\nabla u(t_0)\|_2^2=0$, which implies $u(t_0)=0$.
Let $u(x,t)\equiv0$ for all $t>t_0$, then it is seen that $u(x, t)$ is a weak solution of \eqref{1.1} that vanishes in finite time.
The proof is complete.
\end{proof}

\begin{theorem}\label{th4.2}
Assume $a, b>0$, $3<q\leq2^*-1$, and $u$ is the weak solution of Problem \eqref{1.1} with $u_0\in H_0^1(\Omega)$. If $J(u_0)=d$
and $I(u_0)<0$, then there exists a finite time $T$ such that $u$ blows up at $T$.
\end{theorem}
\begin{proof}
In accordance with the proof of Theorem \ref{th3.2}, by a series of computation, we can get
\begin{eqnarray}\label{4.2}
&&M''(t)M(t)-\frac{q+1}{2}M'(t)^2 \nonumber\\
&\geq&a(q-1)\lambda_1M'(t)M(t)-2(q+1)J(u_0)M(t)-(q+1)\|u_0\|^2_2M'(t).
\end{eqnarray}
Since $J(u_0)=d$, $I(u_0)<0$, and by the continuity of $J(u)$ and $I(u)$ with respect to $t$, there exists a $t_0>0$
such that $J(u(x, t))>0$ and $I(u(x,t))<0$ for $0<t\leq t_0$. Then from $(u_t, u)=-I(u)$, we have $u_t\not\equiv0$
for $0<t\leq t_0$. We get
\begin{eqnarray*}
J(u(t_0))\leq d-\int_0^{t_0}\|u_{\tau}\|_2^2\mathrm{d}\tau=d_1<d.
\end{eqnarray*}
Similarly, choosing $t=t_0$ as the initial time and by Lemma \ref{le2.5}, we know that $u(x, t)\in V_{\delta}$
for $\delta_1<\delta<\delta_2$ and $t>t_0$, where $\delta_1<1<\delta_2$ are the two roots of the equation $d(\delta)=d_1$.
Therefore, we have $I_{\delta}(u)<0$ and $\|\nabla u\|_2>r(\delta)$ for $\delta_1<\delta<\delta_2$ and $t>t_0$. Thus,
$I_{\delta_2}(u)\leq0$ and $\|\nabla u\|_2\geq r(\delta_2)$ for $t>t_0$. Then for $t>t_0$ we get the following estimates
\begin{eqnarray*}
M''(t)&=&-2I(u)=2a(\delta_2-1)\|\nabla u\|_2^2+2b(\delta_2-1)\|\nabla u\|_2^4-2I_{\delta_2}I(u)\\
&\geq&2a(\delta_2-1)\|\nabla u\|_2^2\geq2a(\delta_2-1)r^2(\delta_2), \\
M'(t)&\geq&2a(\delta_2-1)r^2(\delta_2)t,\\
M(t)&\geq&a(\delta_2-1)r^2(\delta_2)t^2.
\end{eqnarray*}
Consequently, for sufficiently large $t$, we get from \eqref{4.2} that
\begin{eqnarray*}
&&M''(t)M(t)-\frac{q+1}{2}M'(t)^2\nonumber\\
&\geq&\Big(\frac{a(q-1)\lambda_1}{2}M(t)-(q+1)\|u_0\|_2^2\Big)M'(t)+\Big(\frac{a(q-1)\lambda_1}{2}M'(t)-2(q+1)d\Big)M(t)>0.
\end{eqnarray*}
The reminder of the proof is the same as that of Theorem \ref{th3.2}.
\end{proof}

\par
\section{$\mathrm{The \ case \ J(u_0)>d}$.}
\setcounter{equation}{0}

In this section, we investigate the conditions to ensure the existence of global or finite time blow-up solutions to
Problem \eqref{1.1} with $J(u_0)>d$.
Inspired by some ideas from \cite{Gazzola,RZXu2013}, where a class of semilinear parabolic and pseudo-parabolic equations were studied, respectively,
we give some sufficient conditions for the solutions to exist globally or not with arbitrarily high initial energy.
For this, set
\begin{eqnarray*}
&& \mathcal{N}_+=\{u\in H_0^1(\Omega)| \ I(u)>0\},\\
&& \mathcal{N}_-=\{u\in H_0^1(\Omega)| \ I(u)<0\},
\end{eqnarray*}
and the (open) sublevels of $J$
\begin{eqnarray*}
&& J^s=\{u\in H_0^1(\Omega)| \ J(u)<s\}.
\end{eqnarray*}
Furthermore, we define
\begin{eqnarray}\label{52.1}
&& \mathcal{N}_s=\mathcal{N}\cap J^s=\Big\{u\in \mathcal{N}\Big|\dfrac{a(q-1)}{2(q+1)}\|\nabla u\|_2^2+\frac{b(q-3)}{4(q+1)}\|\nabla u\|_2^4<s\Big\}.
\end{eqnarray}
By the definition of $d$, we see that for any $s>d$, $\mathcal{N}_s$ is nonempty.
For all $s>d$, set
\begin{equation}\label{52.2}
\lambda_s=\inf\{\|u\|_2\mid u\in \mathcal{N}_s\},\quad\Lambda_s=\sup\{\|u\|_2\mid u\in \mathcal{N}_s\}.
\end{equation}
It is clear that $\lambda_s$ is nonincreasing in $s$ and $\Lambda_s$ are nondecreasing in $s$.

Finally we introduce the following sets
\begin{eqnarray*}
&&\mathcal{B}=\{u_0\in H_0^1(\Omega)\mid \text{the solution} u=u(t)\ \text{of \eqref{1.1} blows up in finite time}\},\\
&&\mathcal{G}=\{u_0\in H_0^1(\Omega)\mid \text{the solution} u=u(t)\ \text{of \eqref{1.1} exists for all}\ t>0\},\\
&&\mathcal{G}_0=\{u_0\in H_0^1(\Omega)\mid u(t)\rightarrow 0\ \text{in}\ H_0^1(\Omega)\ \text{as}\ t\rightarrow\infty\}.
\end{eqnarray*}

The following two lemmas will be needed in the proof of the main results in this section.

\begin{lemma}\label{lem52.1} Let $3<q\leq 2^*-1$. Then

(i) $0$ is away from both $\mathcal{N}$ and $\mathcal{N}_-$, i.e. $dist(0,\mathcal{N})>0$, $dist(0,\mathcal{N}_-)>0$.

(ii) For any $s>0$, the set $J^s\cap \mathcal{N}_+$ is bounded in $H_0^1(\Omega)$.
\end{lemma}
\begin{proof}
(i) For any $u\in \mathcal{N}$, by the definition of $d$ we have
\begin{eqnarray*}
d&\leq&\frac{a}{2}\|\nabla u\|_2^2+\frac{b}{4}\|\nabla u\|_2^4-\frac{1}{q+1}\|u\|^{q+1}_{q+1}\\
 &=&\frac{a}{2}\|\nabla u\|_2^2+\frac{b}{4}\|\nabla u\|_2^4-\frac{1}{q+1}\Big(a\|\nabla u\|_2^2+b\|\nabla u\|_2^4\Big)\\
 &=&\dfrac{a(q-1)}{2(q+1)}\|\nabla u\|_2^2+\dfrac{b(q-3)}{4(q+1)}\|\nabla u\|_2^4.
\end{eqnarray*}
Recalling that $q>3$, the above inequality implies that there exists a constant $c_0>0$ such that $dist(0,\mathcal{N})=\inf_{u\in \mathcal{N}}\|\nabla u\|_2\geq c_0$.

For any $u\in \mathcal{N}_-$, we have $\|\nabla u\|_2\neq0$. Then it follows that
$$a\|\nabla u\|_2^2<a\|\nabla u\|_2^2+b\|\nabla u\|_2^4<\|u\|^{q+1}_{q+1}\leq S^{q+1}\|\nabla u\|_2^{q+1},$$
which implies
$$\|\nabla u\|_2\geq\Big(\dfrac{a}{S^{q+1}}\Big)^{\frac{1}{q-1}}.$$
Here $S>0$ is given in Lemma \ref{depth}.
Therefore, $dist(0,\mathcal{N}_-)=\inf\limits_{u\in \mathcal{N}_-}\|\nabla u\|_2>0$.

(ii) For any $u\in J^s\cap \mathcal{N}_+$, it holds that $J(u)<s$ and $I(u)>0$. Therefore,
\begin{eqnarray*}
s>J(u)&=&\dfrac{a(q-1)}{2(q+1)}\|\nabla u\|_2^2+\dfrac{b(q-3)}{4(q+1)}\|\nabla u\|_2^4+\dfrac{1}{q+1}I(u)\\
&>&\dfrac{a(q-1)}{2(q+1)}\|\nabla u\|_2^2,
\end{eqnarray*}
which yields
$$\|\nabla u\|_2^2\leq \dfrac{2(q+1)s}{a(q-1)}.$$
The proof is complete.
\end{proof}

\begin{lemma}\label{lem52.2}
Let $3<q<2^*-1$. Then for any $s>d$, $\lambda_s$ and $\Lambda_s$ defined in \eqref{52.2} satisfy
\begin{equation}\label{52.3}
0<\lambda_s\leq\Lambda_s<+\infty.
\end{equation}
\end{lemma}

\begin{proof}
By the Gagliardo-Nirenberg inequality, we have
\begin{equation}\label{52.4}
\|u\|_{q+1}^{q+1}\leq C\|\nabla u\|^{n(q-1)/2}_2\|u\|^\alpha_2,\quad\forall\ u\in H_0^1(\Omega),
\end{equation}
where $C$ is a positive constant depending only on $n$ and $q$ and $\alpha=q+1-\dfrac{n(q-1)}{2}>0$ since $q<2^*-1$.
Therefore, for any $s>d$ and $u\in \mathcal{N}_s$, we have
\begin{equation}\label{52.5}
a\|\nabla u\|^{2}_2<\|u\|_{q+1}^{q+1}\leq C\|\nabla u\|^{n(q-1)/2}_2\|u\|^\alpha_2,
\end{equation}
which can be rewritten as
\begin{equation}\label{52.6}
\|\nabla u\|^{2-n(q-1)/2}_2\leq \dfrac{C}{a}\|u\|^\alpha_2.
\end{equation}
By combining Lemma \ref{lem52.1}(i) with \eqref{52.1} we see the left-hand side of \eqref{52.6} remains
bounded away from $0$ no matter what the sign of $2-n(q-1)/2$ is. This proves $\lambda_s>0$
by the definition of $\lambda_s$. Moreover, the fact that $\Lambda_s<\infty$ just follows from \eqref{52.1}
and Poincar\'{e}'s inequality $\|u\|_2\leq C_*\|\nabla u\|_2$. The proof is compete.
\end{proof}

\begin{remark}
The condition that $q<2^*-1$ is only required when showing the positivity of $\lambda_s$ for $s>d$.
\end{remark}

To give some sufficient conditions for the existence of global and blow-up solutions for supercritical initial energy, 
denote by $T(u_0)$ the maximal existence time of the solutions to Problem \eqref{1.1}
with initial datum $u_0$. If the solution is global, i.e. $T(u_0)=\infty$, we denote by
$$\omega(u_0)=\bigcap_{t\geq0}\overline{\{u(s):\ s\geq t\}}^{H_0^1(\Omega)}$$
the $\omega$-limit set of $u_0$. The main result of this section is the following

\begin{theorem}\label{th5.1}
Let $3<q<2^*-1$. Assume that $J(u_0)>d$, then the following statements hold

(i) If $u_0\in \mathcal{N}_+$ and $\|u_0\|_2\leq\lambda_{J(u_0)}$,
then $u_0\in\mathcal{G}_0$;

(ii) If $u_0\in \mathcal{N}_-$ and $\|u_0\|_2\geq\Lambda_{J(u_0)}$,
then $u_0\in\mathcal{B}$.
\end{theorem}

\begin{proof}
(i) Assume that $u_0\in \mathcal{N}_+$ satisfying $\|u_0\|_2\leq \lambda_{J(u_0)}$. We first claim that
$u(t)\in \mathcal{N}_+$ for all $t\in[0,T(u_0))$. If not, there exists $t_0\in(0,T(u_0))$ such that $u(t)\in \mathcal{N}_+$
for $0\leq t<t_0$ and $u(t_0)\in \mathcal{N}$. On the other hand, it follows from \eqref{2.2} that $J(u(t_0))\leq J(u_0)$,
which implies that $u(t_0)\in J^{J(u_0)}$. Therefore, $u(t_0)\in \mathcal{N}_{J(u_0)}$. According to the definition of $\lambda_{J(u_0)}$,
we have
\begin{equation}\label{53.1}
\|u(t_0)\|_2\geq \lambda_{J(u_0)}.
\end{equation}
Taking $\phi=u$ in \eqref{2.1}, we get
\begin{eqnarray}\label{53.2}
\frac{1}{2}\frac{d}{dt}\|u\|_2^2=(u_t, u)=-a\|\nabla u\|_2^2-b\|\nabla u\|_2^4+\|u\|_{q+1}^{q+1}=-I(u).
\end{eqnarray}
Recalling that $I(u(t))>0$ for $t\in[0,t_0)$, we obtain from \eqref{53.2} that
\begin{equation*}
\|u(t_0)\|_2<\|u_0\|_2\leq \lambda_{J(u_0)},
\end{equation*}
which is contradictive with \eqref{53.1}. So $u(t)\in \mathcal{N}_+$ and this in turn implies that $u(t)\in J^{J(u_0)}$ for all $t\in[0,T(u_0))$.
Lemma \ref{lem52.1} (ii) shows that the orbit $\{u(t)\}$ remains bounded in $H_0^1(\Omega)$ for $t\in[0,T(u_0))$ so that $T(u_0)=\infty$.
Let $\omega$ ba an arbitrary element in $\omega(u_0)$, then by \eqref{2.2} and \eqref{53.2} we have
$$\|\omega\|_2<\lambda_{J(u_0)},\quad J(\omega)\leq J(u_0),$$
which, recalling the definition of $\lambda_{J(u_0)}$ again, implies $\omega(u_0)\cap \mathcal{N}=\emptyset$.
Therefore, $\omega(u_0)=\{0\}$, i.e. $u_0\in\mathcal{G}_0$.

(ii) Assume that $u_0\in \mathcal{N}_-$ with $\|u_0\|_2\geq \Lambda_{J(u_0)}$. We first claim that
$u(t)\in \mathcal{N}_-$ for all $t\in[0,T(u_0))$. If not, there exists $t^0\in(0,T(u_0))$ such that $u(t)\in \mathcal{N}_-$
for $0\leq t<t^0$ and $u(t^0)\in \mathcal{N}$. Noticing \eqref{2.2} we have $J(u(t^0))\leq J(u_0)$,
which implies that $u(t^0)\in J^{J(u_0)}$. Therefore, $u(t^0)\in \mathcal{N}_{J(u_0)}$.
According to the definition of $\Lambda_{J(u_0)}$,
we have
\begin{equation}\label{53.3}
\|u(t^0)\|_2\leq \Lambda_{J(u_0)}.
\end{equation}
On the other hand, from \eqref{53.2} and the fact that $I(u(t))<0$ for $t\in[0,t^0)$, we get
\begin{equation*}
\|u(t^0)\|_2>\|u_0\|_2\geq \Lambda_{J(u_0)},
\end{equation*}
a contradiction with \eqref{53.3}.

If $T(u_0)=\infty$, then for every $\omega\in\omega(u_0)$, it follows from \eqref{2.2} and \eqref{53.2}
that
\begin{equation}\label{53.4}
\|\omega\|_2>\Lambda_{J(u_0)},\quad J(\omega)\leq J(u_0).
\end{equation}
Combining \eqref{3.4} with the definition of $\Lambda_{J(u_0)}$ again, we obtain  $\omega(u_0)\cap \mathcal{N}=\emptyset$.
Thus, it must hold that $\omega(u_0)=\{0\}$, which is contradictive with Lemma \ref{lem52.1}(i).
Hence, $T(u_0)<\infty$ and the proof is complete.
\end{proof}

Theorem \ref{th5.1} (ii) implies that there exists $u_0$ such that $J(u_0)$ is arbitrarily large,
while the corresponding solution $u(x,t)$ of Problem \eqref{1.1} with $u_0$ as initial datum blows up in finite time.
To illustrate this, we need the following proposition.

\begin{proposition}\label{pro5.1}
Let $3<q\leq 2^*-1$ and $J(u_0)>d$.
If $\dfrac{4(q+1)}{q-3}|\Omega|^{\frac{q-1}{2}}J(u_0)\leq \|u_0\|_2^{q+1}$, then $u_0\in \mathcal{N}_-\cap\mathcal{B}$.
\end{proposition}

\begin{proof}
By using H\"{o}lder's inequality we obtain from $\dfrac{4(q+1)}{q-3}|\Omega|^{\frac{q-1}{2}}J(u_0)\leq \|u_0\|_2^{q+1}$
that
\begin{equation}\label{53.5}
\dfrac{4(q+1)}{q-3}|\Omega|^{\frac{q-1}{2}}J(u_0)\leq \|u_0\|_2^{q+1}\leq\|u_0\|_{q+1}^{q+1}|\Omega|^{\frac{q-1}{2}}.
\end{equation}
By combining the expression of $J(u_0)$, $I(u_0)$ with \eqref{53.5} we have
\begin{eqnarray*}
J(u_0)&=&\frac{a}{2}\|\nabla u_0\|_2^2+\frac{b}{4}\|\nabla u_0\|_2^4-\frac{1}{q+1}\|u_0\|^{q+1}_{q+1},\\
&=&\frac{a}{4}\|\nabla u_0\|_2^2+(\frac{1}{4}-\frac{1}{q+1})\|u_0\|^{q+1}_{q+1}+\frac{1}{4}I(u_0)\\
&>&\dfrac{q-3}{4(q+1)}\|u_0\|^{q+1}_{q+1}+\frac{1}{4}I(u_0)\\
&\geq&J(u_0)+\frac{1}{4}I(u_0),
\end{eqnarray*}
which shows that $I(u_0)<0$, i.e. $u_0\in \mathcal{N}_-$.

To show that $u_0\in \mathcal{B}$, we need only to prove that $\|u_0\|_2\geq\Lambda_{J(u_0)}$ by Theorem \ref{th5.1}.
For this, $\forall\ u\in \mathcal{N}_{J(u_0)}$, we have
\begin{eqnarray*}
\|u\|_2^{q+1}&\leq&|\Omega|^{\frac{q-1}{2}}\|u\|_{q+1}^{q+1}=|\Omega|^{\frac{q-1}{2}}(a\|\nabla u\|_2^2+b\|\nabla u\|_2^4)\\
&=&|\Omega|^{\frac{q-1}{2}}\dfrac{4(q+1)}{q-3}\Big\{(\dfrac{1}{4}-\dfrac{1}{q+1})a\|\nabla u\|_2^2+(\dfrac{1}{4}-\dfrac{1}{q+1})b\|\nabla u\|_2^4\Big\}\\
&\leq&|\Omega|^{\frac{q-1}{2}}\dfrac{4(q+1)}{q-3}\Big\{(\dfrac{1}{2}-\dfrac{1}{q+1})a\|\nabla u\|_2^2+(\dfrac{1}{4}-\dfrac{1}{q+1})b\|\nabla u\|_2^4\Big\}\\
&<&|\Omega|^{\frac{q-1}{2}}\dfrac{4(q+1)}{q-3}J(u_0)\leq \|u_0\|_2^{q+1}.
\end{eqnarray*}
Taking supermum over $\mathcal{N}_{J(u_0)}$ we obtain
$$\Lambda^{q+1}_{J(u_0)}\leq |\Omega|^{\frac{q-1}{2}}\dfrac{4(q+1)}{q-3}J(u_0)\leq\|u_0\|_2^{q+1},$$
i.e. $\|u_0\|_2\geq\Lambda_{J(u_0)}$. Therefore, $u_0\in \mathcal{N}_-\cap\mathcal{B}$. The proof is complete.
\end{proof}

\begin{theorem}\label{th5.2}
For any $M>d$, there exists $u_M\in \mathcal{N}_-$ such that $J(u_M)\geq M$ and $u_M\in \mathcal{B}$.
\end{theorem}

\begin{proof}
For any $M>d$, let $\Omega_1$ and $\Omega_2$ be two arbitrary disjoint open subdomains of $\Omega$,
and assume that $v\in H_0^1(\Omega_1)$ is an arbitrary nontrivial function. Since $q>3$,
we can choose $\alpha>0$ large enough such that $J(\alpha v)\leq0$ and $\|\alpha v\|_2^{q+1}>|\Omega|^{\frac{q-1}{2}}\dfrac{4(q+1)}{q-3}M$.
Fix $\alpha$ and choose a function $w\in H_0^1(\Omega_2)$ such that $J(w)+J(\alpha v)=M$.
Extend $v$ and $w$ to be $0$ in $\Omega\setminus\Omega_1$ and $\Omega\setminus\Omega_2$, respectively,
and set $u_M=\alpha v+w$. Then $J(u_M)=J(\alpha v)+J(w)=M$ and it holds that $\|u_M\|_2^{q+1}\geq\|\alpha v\|_2^{q+1}>|\Omega|^{\frac{q-1}{2}}\dfrac{4(q+1)}{q-3}J(u_M)$.
By Proposition \ref{pro5.1} it is seen that $u_M\in \mathcal{N}_-\cap\mathcal{B}$. This completes the proof.
\end{proof}

{\bf Acknowledgement}\\
The authors would like to express their sincere gratitude to Professor Wenjie Gao for his enthusiastic
guidance and constant encouragement.

\end{document}